\documentclass[12pt]{amsart}
\usepackage{amsmath,amsfonts,amssymb,amscd,amsthm}
\usepackage {pstricks}
\usepackage{pstricks,pst-node}

\newtheorem*{theorem*}{Theorem}
\newtheorem*{lemma*}{Lemma}
\newtheorem{theorem}[equation]{Theorem}
\newtheorem{proposition}[equation]{Proposition}
\newtheorem{definition}[equation]{Definition}
\newtheorem{corollary}[equation]{Corollary}
\theoremstyle{remark}
\newtheorem{remark}[equation]{Remark}

\DeclareMathOperator{\Hom}{Hom}
\DeclareMathOperator{\Ext}{Ext}
\DeclareMathOperator{\Ind}{Ind}

\DeclareMathOperator{\rank}{rank}
\newcommand{\bbF}{{\mathbb F}}

\newcommand{\C}{{\mathbb C}}
\newcommand{\R}{{\mathbb R}}
\newcommand{\Q}{{\mathbb Q}}
\newcommand{\Z}{{\mathbb Z}}
\newcommand{\N}{{\mathbb N}}

\newcommand{\bbP}{{\mathbb P}}

\newcommand{\CT}{{\mathcal T}}

\newcommand{\Sym}{{\mathrm{Sym}}}

\newcommand{\inv}{^{-1}}

\newcommand{\lexp}[2]{\kern\scriptspace\vphantom{#2}^{#1}\kern-\scriptspace#2}

\newcommand\be{\begin{equation}}
\newcommand\ee{\end{equation}}

\newcommand\tr{\mathrm{tr}}

\begin{document}
\author{Anthony~Henderson and Gus~Lehrer}
\address{
School of Mathematics and Statistics\\
University of Sydney\\
NSW 2006, Australia\\
Fax: +61 2 9351 4534}
\email{anthonyh@maths.usyd.edu.au, gusl@maths.usyd.edu.au}
\subjclass[2000]{Primary 14M25, 14P25; Secondary 20C30, 14L30}
\title [Real Coxeter toric varieties]
{The equivariant Euler characteristic of real Coxeter toric varieties}
\begin{abstract}
Let $W$ be a Weyl group, and let $\CT_W$ be
the complex toric variety attached to the fan of cones
corresponding to the reflecting hyperplanes of $W$, 
and its weight lattice.
The real locus $\CT_W(\R)$ is a smooth, connected, compact manifold with
a $W$-action.
We give a formula for the equivariant Euler characteristic of
$\CT_W(\R)$ as a generalised character of $W$. In type $A_{n-1}$ for $n$ odd,
one obtains a generalised character of $\Sym_n$ whose degree
is (up to sign) the $n^{\text{th}}$ Euler number.
\end{abstract}

\maketitle
\section{Introduction} 
Let $\Phi$ be a crystallographic root system 
in a Euclidean space $V$, and let $\Pi$ be a simple subsystem of $\Phi$.
We assume that $\Phi$ spans $V$,
and write $r=|\Pi|=\dim V$.
Let $W$ be the Weyl group of $\Phi$. As described in \cite{procesi},
and more generally in \cite{F}, there is a 
complex toric variety $\CT_W$ associated to this data, which is defined by the fan
of rational convex polyhedral cones into which $V$ is divided by
the reflecting hyperplanes of $W$, and the weight lattice 
$N=\{\omega\in V\mid\langle\omega,\alpha\rangle\in\Z,\,\forall\alpha\in \Phi\}$.
As shown in \cite[Section IV]{dps}, $\CT_W$ coincides with the Hessenberg variety
in which the chosen subset of negative roots is $-\Pi$.

The Weyl group $W$ clearly acts on $\CT_W$, and hence acts on its rational
cohomology $H^\bullet(\CT_W,\Q)$. This action has been studied in several
papers, for example \cite{procesi}, \cite{stembridge}, \cite{L}. 

In this paper we consider instead the real locus $\CT_W(\R)$, which is
the set of fixed points of the complex conjugation $\sigma$ on $\CT_W$.

\begin{proposition}\label{twr}
\begin{enumerate}
\item
$\CT_W(\R)$ is a smooth, connected, compact manifold 
of dimension $r$.
\item
$\CT_W(\R)$ has a real cell decomposition in which the cells $C_w$ are indexed
by $w\in W$, and $C_w\cong\R^{d(w)}$ where
$$
d(w)=|\{\alpha\in\Pi\,|\,w(\alpha)\in\Phi^-\}|.
$$
\end{enumerate}
\end{proposition}
\begin{proof}
In part (\textrm{i}), 
everything but the connectedness follows from the fact that $\CT_W$ is a nonsingular
projective variety of dimension $r$
(see, for example, \cite[Chapter II, \S2.3]{shafarevich} -- it is 
trivial that $\CT_W(\R)$ is nonempty).
Connectedness follows from part (\textrm{ii}), since there is a unique 
$0$-dimensional cell (more generally, the real locus of any nonsingular 
projective toric variety
is connected, because one can use a filtration defined
as in \cite[\S5.2]{F} instead of the cell decomposition).
The analogue of part (\textrm{ii}) for the complex variety may be found in \cite{dps},
and it is clear that the complex cells are defined over $\R$, i.e.\ are fixed by $\sigma$.
\end{proof}

\begin{corollary} \label{euler}
The Euler characteristic of $\CT_W(\R)$ is given by
$$
\chi(\CT_W(\R))=\sum_{w\in W}(-1)^{d(w)},
$$
and is zero if $r$ is odd.
\end{corollary}
\begin{proof}
The formula for Euler characteristic follows immediately from part (\textrm{ii}) of
Proposition \ref{twr}. Any odd-dimensional compact manifold has zero Euler characteristic.
\end{proof}

Our aim is to give a description of the Euler characteristic which is equivariant,
i.e.\ incorporates the $W$-action on $\CT_W(\R)$. The cohomology spaces
$H^i(\CT_W(\R),\Q)$ are $\Q W$-modules, and so may be considered as elements 
of the Grothendieck ring $R(W)=R(\Q W)$. 
\begin{definition} \label{deflambda}
The equivariant Euler characteristic of $\CT_W(\R)$ is the following
element of $R(W)$:
$$
\Lambda_W=\sum_i (-1)^i H^i(\CT_W(\R),\Q).
$$
\end{definition}

\begin{definition} \label{defpi}
Let $\pi_W^{(2)}$ be the permutation character of $W$ on the finite set $N/2N$,
where $N$ is the weight lattice as above.
For any subset $J\subseteq\Pi$, let $\pi_{W_J}^{(2)}$ be the corresponding
character of the parabolic subgroup $W_J$, regarded as the Weyl group of the parabolic
subsystem $\Phi_J$ in the subspace $V_J=\R J$.
\end{definition}


In Section 2, we will prove the following formula.
\begin{theorem}\label{main}
The equivariant Euler characteristic $\Lambda_W$ of $\CT_W(\R)$
is given as an element of the Grothendieck ring $R(W)$ by
$$
\Lambda_W=\varepsilon_W\sum_{J\subseteq \Pi}(-1)^{|J|}\,\Ind_{W_J}^W(\pi_{W_J}^{(2)}),
$$
where $\varepsilon_W$ denotes the sign character of $W$.
\end{theorem}

In Section 3, we will deduce the following more explicit formula in type A.
\begin{theorem}\label{typea}
Suppose that $\Phi$ is of type A$_{n-1}$, so that $W\cong\Sym_n$. Then
\[
\Lambda_{\Sym_n}=
\sum_{m\geq 0}(-1)^m
\negthickspace\negthickspace
\sum_{\substack{n_1,n_2,\cdots,n_m\geq 2\\
n_1,n_2,\cdots,n_m\textup{ even}\\n_1+n_2+\cdots+n_m\leq n}}
\negthickspace\negthickspace
\Ind_{\Sym_{n-n_1-\cdots-n_m}\times \Sym_{n_1}\times
\cdots \times \Sym_{n_m}}^{\Sym_n}(\varepsilon_{n_1,\cdots,n_m}),
\]
where $\varepsilon_{n_1,\cdots,n_m}$ is the linear character
whose restriction to the $\Sym_{n-n_1-\cdots-n_m}$ factor is trivial and
whose restriction to the $\Sym_{n_i}$ factors is the sign character.
\end{theorem}
\noindent
An intriguing question suggested by the form of
Theorem \ref{typea} is whether the inner sum
equals the individual cohomology space $H^m(\CT_{\Sym_n}(\R),\Q)$
as $\Sym_n$-module.

\begin{remark}\label{cell-orbit}
Theorem \ref{main} implies the following formula for the
nonequivariant Euler characteristic:
\be
\chi(\CT_W(\R))=\sum_{J\subseteq \Pi}[W:W_J]\,(-2)^{|J|}.
\ee
The fact that this equals the formula given in Corollary \ref{euler}
is the $q=-1$ case of the following identity, alluded to in
\cite[Remark 3.6]{L}:
\be \label{poincare}
\sum_{J\subseteq \Pi}[W:W_J]\,(q-1)^{|J|}=\sum_{w\in W}q^{d(w)}.
\ee
The two sides of \eqref{poincare} are the expressions for the Poincar\'e
polynomial of $\CT_W$ resulting from the viewpoints of \cite{L} and \cite{dps}
respectively. There is an elementary proof of \eqref{poincare}:
rewrite $[W:W_J]$ as $|\{w\in W\,|\,w(J)\subseteq\Phi^-\}|$, expand
$(q-1)^{|J|}$ using the Binomial Theorem, and apply the Inclusion-Exclusion Principle.
\end{remark}
In Section 4, we will compare
the equivariant Euler characteristic of $\CT_W(\R)$
with the $q=-1$ specialization of the equivariant Poincar\'e polynomial
of $\CT_W$, and thus place our results in a more general context.
\section{Proof of Theorem \ref{main}}

Since the cell decomposition of $\CT_W$ defined in \cite{dps} 
is not preserved under the action of $W$, it is not convenient
for the computation of the equivariant Euler characteristic. Instead we
use the decomposition into torus orbits, as in \cite{L}.
Although we will not need it, we remark that it is possible to 
give a $W$-stable CW-complex construction of $\CT_W(\R)$, using
\cite[Theorem 2.4]{delaunay}.

Let $G$ be the semisimple complex algebraic group corresponding to
the root datum $(\Z\Phi,\Phi,N,\Phi^\vee)$ (thus, $G$ is of adjoint type),
and let $T$ be a maximal torus of $G$ with cocharacter group $Y(T)\cong N$.
By construction, $\CT_W$ carries an action of $T$.

Since $T$ is defined over $\Z$, it has a canonical real structure.
The Weyl group elements give
automorphisms of $T$ which preserve the real locus $T(\R)$. 
\begin{definition} \label{defphi}
Define the generalised character $\phi_W\in R(W)$ by
$$
\phi_W(w)=\chi_c(T(\R)^w),
$$
where $\chi_c$ denotes the compact-supports Euler characteristic.
For any $J\subseteq\Pi$, let $\phi_{W_J}$ be the corresponding
element of $R(W_J)$.
\end{definition}

\begin{proposition}\label{sumind}
In $R(W)$, the following equation holds:
$$
\Lambda_W=\sum_{J\subseteq \Pi}\Ind_{W_J}^W\phi_{W_J}.
$$
\end{proposition}
\begin{proof}
By the Lefschetz fixed-point formula, the trace of an element $w\in W$
on $\Lambda_W$ is given by $\chi(\CT_W(\R)^w)$.
Since $\CT_W(\R)$ is compact, we may replace $\chi$ with the compact-supports
Euler characteristic $\chi_c$, which is an ``Eulerian function''
in the sense of \cite{dl} (i.e.\ additive
over decompositions of spaces into locally closed subspaces).

As noted in \cite[proof of Theorem 1.1]{L}, $\CT_W$ is the 
disjoint union of the orbits of $T$ on $\CT$; these orbits
are tori, so that $\CT_W=\coprod_{xW_J}T(xW_J)$, where the union 
is over all cosets $xW_J$ of all parabolic subgroups $W_J$
($J\subseteq \Pi$). 
We have $\dim (T(xW_J))=|J|$, and 
$Y(T(xW_J))=xM_J$, where $M_J$ is the weight lattice of
$W_J$ as in \S 1. 
This decomposition is stable 
under complex conjugation $\sigma$, and the pieces are
permuted by $W$: $w\cdot T(xW_J)=T(wxW_J)$, 
whence it follows that 
$$
\CT_W(\R)^w=\coprod_{\substack{xW_J\\wxW_J=xW_J}}T(xW_J)(\R)^w,
$$
and by the additivity of the compact-supports Euler characteristic,
$$
\chi_c(\CT_W(\R)^w)=\sum_{\substack{xW_J\\wxW_J=xW_J}}\chi_c(T(xW_J)(\R)^w).
$$
But if $wxW_J=xW_J$, then $T(xW_J)^w=x\cdot T(W_J)^{x\inv wx}$. So we obtain
$$
\chi_c(\CT_W(\R)^w)=\sum_{\substack{xW_J\\wxW_J=xW_J}}\chi_c(T(W_J)(\R)^{x\inv wx})
=\sum_{\substack{xW_J\\wxW_J=xW_J}}\phi_{W_J}(x\inv wx).
$$
The Proposition now follows from Frobenius' formula for induced characters.
\end{proof}

In view of Proposition \ref{sumind},
Theorem \ref{main} follows from the following result, applied to every parabolic
subgroup $W_J$ of $W$.

\begin{proposition}\label{phipsi}
We have $\phi_W=(-1)^r\,\epsilon_W\,\pi_W^{(2)}$.
\end{proposition}
\begin{proof}

We have an isomorphism of topological groups
$T(\R)\cong Y(T)\otimes_\Z \R^\times$,
where $\R^\times$ is viewed as a $\Z$-module via its abelian group structure, and
the topology on $Y(T)\otimes_\Z \R^\times$ is such that $\{0\}\otimes \R^\times$
is an open and closed subgroup homeomorphic to $\R^\times$.
For any $w\in W$, the action of $w$ on $T(\R)$ corresponds to the action of
$w\otimes\mathrm{id}$ on $Y(T)\otimes_\Z \R^\times$.
Moreover, we have another isomorphism of topological groups
$\R\oplus\Z/2\Z\overset{\sim}{\to}
\R^\times$, defined by $(x,a)\mapsto(-1)^a e^x$. Hence
$$
T(\R)\cong Y(T)\otimes_\Z(\R\oplus\Z/2\Z)\cong (Y(T)\otimes_\Z \R)\oplus
(Y(T)\otimes_\Z \Z/2\Z)\cong V\oplus N/2N,
$$
where $V$ has its usual topology, and $N/2N$ is a finite discrete set.
This isomorphism respects the action of $W$ on both sides, 
$W$ acting diagonally on the right side.

It follows that for any $w\in W$, we have a homeomorphism
$T(\R)^w\cong V^w \times (N/2N)^w$,
and hence
$$
\chi_c(T(\R)^w)=\chi_c(V^w)\, |(N/2N)^w|=(-1)^{\dim V^w}\pi_W^{(2)}(w),
$$
because $V^w$ is a real vector space.
The result now follows from the fact that $\varepsilon_W(w)=(-1)^{r-\dim V^w}$.
\end{proof}

\section{Type A}

In this section we restrict attention to the case where $\Phi$ is of type
A$_{n-1}$, so that $W\cong\Sym_n$, the symmetric group of degree $n\geq 1$.
In this case, the quantity $d(w)$ defined in Proposition \ref{twr} is the
number of descents of the permutation $w$, and it follows from Corollary
\ref{euler} that
\be \label{tan}
\chi(\CT_{\Sym_n}(\R))=\begin{cases}
0,&\text{ if $n$ is even,}\\
(-1)^{\frac{n-1}{2}}E_n,&\text{ if $n$ is odd,}
\end{cases}
\ee
where $E_n$ is the Euler number, i.e.\ the coefficient of $\frac{x^n}{n!}$
in the Taylor series of $\tan(x)$ (see \cite[Section 3.16]{stanley}).

%
%

We have a simple expression for the permutation character
$\pi_{\Sym_n}^{(2)}$.

\begin{proposition} \label{piformula}
For any $n\geq 1$, the following equation holds in $R(\Sym_n)$:
$$
\pi_{\Sym_n}^{(2)}=\sum_{\substack{0\leq s\leq n\\s\textup{ even}}}
\Ind_{\Sym_{n-s}\times\Sym_s}^{\Sym_n}(1).
$$
\end{proposition}
\begin{proof}
If $n$ is odd, this reflects an isomorphism of sets with an action of
$\Sym_n$: the subgroups $\Sym_{n-s}\times\Sym_s$ referred to in the
statement are precisely the stabilizers of representatives of the
$\Sym_n$-orbits in $N/2N$. However, this is not the case for $n$ even,
so we shall prove the equality on the level of characters. Since
$\Ind_{\Sym_{n-s}\times\Sym_s}^{\Sym_n}(1)(w)$ is the
number of $w$-stable subsets of $\{1,2,\cdots,n\}$ which have $s$ elements, this
amounts to proving that $\pi_{\Sym_n}^{(2)}(w)$ is the number of
$w$-stable subsets of $\{1,2,\cdots,n\}$ which have an even number of elements,
for any $w\in\Sym_n$.

Now the weight lattice $N$ referred to in Definition \ref{defpi}
may be identified with $\Z^n/\Z(1,1,\cdots,1)$, where $\Sym_n$ acts by
permuting the coordinates. Hence we may identify $N/2N$ with
$(\Z/2\Z)^n/(\Z/2\Z)(1,1,\cdots,1)$. For any $w\in\Sym_n$,
$\pi_{\Sym_n}^{(2)}(w)$ is, by definition, the number of $w$-fixed elements
of this set. Hence
$2\pi_{\Sym_n}^{(2)}(w)$ is the number of elements in the set
$$
\{(a_1,a_2,\cdots,a_n)\in(\Z/2\Z)^n\,|\,
\text{either }a_{w(i)}=a_i,\,\forall i,\text{ or }
a_{w(i)}=a_i+1,\,\forall i\}.
$$
Note that if $w$ has a cycle of odd length, the case
that $a_{w(i)}=a_i+1$ for all $i$ cannot occur. It is easy to deduce the following
formula, in which $c(w)$ denotes the number of cycles of $w$:
\be
\pi_{\Sym_n}^{(2)}(w)=\begin{cases}
2^{c(w)-1}&\text{ if $w$ has a cycle of odd length,}\\
2^{c(w)}&\text{ if all the cycles of $w$ have even length.}
\end{cases}
\ee
Clearly the right-hand side equals the number of
$w$-stable subsets of $\{1,2,\cdots,n\}$ which have an even number of elements,
as required.
\end{proof}

We shall now deduce Theorem \ref{typea} from Proposition \ref{piformula}
and Theorem \ref{main}. It is convenient to consider 
$\bigoplus_{n\geq 0}R(\Sym_n)$, which as usual
we regard as an $\N$-graded ring via the induction product:
\[
\chi_1.\chi_2=\Ind_{\Sym_{n_1}\times\Sym_{n_2}}^{\Sym_{n_1+n_2}}(\chi_1\boxtimes\chi_2),
\text{ for }\chi_1\in R(\Sym_{n_1}),\,\chi_2\in R(\Sym_{n_2}).
\]
(The identity element of this ring is the trivial character of $\Sym_0$, which we
will write simply as $1$.) In fact, to work with all $n$ simultaneously, we need the
completion $\widehat{\bigoplus}_{n\geq 0}R(\Sym_n)$. In this ring, every element with 
degree-$0$ term equal to $1$ has a multiplicative inverse.

The type A case of Theorem \ref{main} says that
\[
\varepsilon_{\Sym_n}\Lambda_{\Sym_n}=\sum_{m\geq 0}(-1)^{n-m}
\negthickspace\negthickspace
\sum_{\substack{n_1,n_2,\cdots,n_m\geq 1\\
n_1+n_2+\cdots+n_m=n}}
\negthickspace\negthickspace
\Ind_{\Sym_{n_1}\times
\cdots \times \Sym_{n_m}}^{\Sym_n}(\pi_{\Sym_{n_1}}^{(2)}\boxtimes\cdots\boxtimes
\pi_{\Sym_{n_m}}^{(2)}).
\]
This can be translated into the following equality in 
$\widehat{\bigoplus}_{n\geq 0}R(\Sym_n)$:
\be \label{firsteqn}
1+\sum_{n\geq 1}(-1)^n\,\varepsilon_{\Sym_n}\Lambda_{\Sym_n}=
\bigg(1+\displaystyle\sum_{n\geq 1}\pi_{\Sym_{n}}^{(2)}\bigg)^{-1}.
\ee
But Proposition \ref{piformula} amounts to the following equality in 
$\widehat{\bigoplus}_{n\geq 0}R(\Sym_n)$:
\be \label{secondeqn}
1+\sum_{n\geq 1}\pi_{\Sym_{n}}^{(2)}=\bigg(\sum_{n\geq 0}1_{\Sym_n}\bigg).
\bigg(\sum_{\substack{n\geq 0\\n\text{ even}}}1_{\Sym_n}\bigg).
\ee
We now substitute \eqref{secondeqn} into \eqref{firsteqn} and use the well-known
fact that the multiplicative inverse of $\sum_{n\geq 0}1_{\Sym_n}$ is
$\sum_{n\geq 0}(-1)^n\,\varepsilon_{\Sym_n}$, to obtain:
\be \label{thirdeqn}
1+\sum_{n\geq 1}(-1)^n\,\varepsilon_{\Sym_n}\Lambda_{\Sym_n}=
\bigg(\sum_{n\geq 0}(-1)^n\,\varepsilon_{\Sym_n}\bigg).
\bigg(1+\displaystyle\sum_{\substack{n\geq 2\\n\text{ even}}}1_{\Sym_n}\bigg)^{-1}.
\ee
Applying the ring involution of $\widehat{\bigoplus}_{n\geq 0}R(\Sym_n)$
which maps $\chi\in R(\Sym_n)$ to $(-1)^n\varepsilon_{\Sym_n}\chi\in R(\Sym_n)$ gives:
\be \label{fourtheqn}
1+\sum_{n\geq 1}\Lambda_{\Sym_n}=
\bigg(\sum_{n\geq 0}1_{\Sym_n}\bigg).
\bigg(1+\displaystyle\sum_{\substack{n\geq 2\\n\text{ even}}}
\varepsilon_{\Sym_n}\bigg)^{-1}.
\ee
Extracting the degree-$n$ term on both sides, we obtain
Theorem \ref{typea}.
\section{Comparison of $\CT_W(\R)$ and $\CT_W$}
If $X$ is a complex variety with a real structure, the rational cohomology
of $X(\R)$ may have little to do with that of $X$. For example,
compare the De~Concini--Procesi models studied in \cite{yuzvinsky}
with their real loci studied in \cite{hendersonrains}.

However, the relationship becomes tighter if one considers not the individual
cohomology spaces $H^i(X(\R),\Q)$, but only the alternating sum. Indeed,
there is a large class of complex varieties $X$
for which the compact-supports Euler characteristic of the real locus 
$X(\R)$ can be obtained from the rational point counting function
$P_X(q)=|X(\bbF_q)|$, if $P_X(q)$ is a polynomial, by setting $q=-1$.
Such varieties $X$ are called computable in \cite[\S 5]{kisinlehrer}.
The condition that $P_X(q)$ be polynomial relates to the structure of 
the cohomology of $X$ (cf. \cite{kisinlehrer}). It applies to 
toric varieties as follows.

Write 
$Gr_{\bf F}^\ell Gr_{\bar{\bf F}}^mH^j_c(X,\C)$
for the $(\ell,m)$-graded part of the Hodge filtration of the cohomology
of the complex variety $X$. It is clear from the results of \cite[\S 2]{kisinlehrer}
that $X$ is computable if 
\be\label{computable} P_X(q)\in\Z[q], \text{ and }
\sigma \text{ acts as }(-1)^\ell\text{ on }Gr_{\bf F}^\ell Gr_{\bar{\bf F}}^\ell H^j_c(X,\C)
\ee
for each $j$ and $\ell$.

The results of \cite{F} and \cite[Proposition 5.2]{kisinlehrer}
make it evident that any toric variety $X$ satisfies
(\ref{computable}). This is because $X$ is a union of
locally closed subvarieties isomorphic to tori, which all satisfy
(\ref{computable}).

In the case of a nonsingular projective toric variety $X$, we have
$H^{2i+1}(X,\Q)=0$ and 
$H^{2i}(X,\Q)\otimes_\Q\C=Gr_{\bf F}^iGr_{\bar{\bf F}}^iH^{2i}_c(X,\C)$, so
$P_X(q)=\sum_i\dim H^{2i}(X,\Q)\,q^i$.
It follows that
\be \label{qtominusone}
\chi(X(\R))=\sum_{i}(-1)^i \dim H^{2i}(X,\Q).
\ee 

To obtain a degree-by-degree statement, one needs to consider
cohomology with coefficients in $\Z/2\Z$.
\begin{proposition} \label{comparison}
Let $X$ be a nonsingular projective complex toric variety.
There is an isomorphism 
\[ H^{2i}(X,\Z/2\Z)\cong H^i(X(\R),\Z/2\Z)\] 
for any $i$, which is equivariant for any automorphism of $X$ defined over $\R$.
\end{proposition}
\begin{proof}
This follows from \cite[Theorem 0.1]{krasnov}, for instance.
See \cite[\S4]{bihanetal} for a summary of the various alternative proofs.
\end{proof}

We deduce an equivariant version of \eqref{qtominusone}.
\begin{corollary} \label{oddorder}
Let $X$ be a nonsingular projective complex toric variety, and
$\alpha$ an automorphism of $X$ of finite odd order, defined over $\R$. Then
\[
\sum_i (-1)^i\, \tr(\alpha,H^i(X(\R),\Q))=\sum_i (-1)^i\, \tr(\alpha,H^{2i}(X,\Q)).
\]
\end{corollary}
\begin{proof}
Let $A$ be the cyclic group of odd order generated by $\alpha$. It suffices
to prove the following equation in the Grothendieck ring $R(A)=R(\Q A)$:
\be \label{groth}
\sum_i (-1)^i\, H^i(X(\R),\Q)=\sum_i (-1)^i\, H^{2i}(X,\Q).
\ee
Since $|A|$ is odd, there is an isomorphism $\tau:R(A)\to R((\Z/2\Z)A)$,
where $R((\Z/2\Z)A)$ denotes the Grothendieck ring of $(\Z/2\Z)A$-modules:
for any $\Q A$-module $M$, $\tau(M)$ is defined by choosing an integral form of $M$
and reducing modulo $2$. So it suffices to prove the equation in
$R((\Z/2\Z)A)$ obtained from \eqref{groth} by applying $\tau$ to both sides. 

Now the universal coefficient theorem for cohomology implies the following equations
in $R((\Z/2\Z)A)$. (Here $S[2]$ denotes the $2$-primary component of
a $\Z$-module $S$.)
\[
\begin{split}
\tau(H^i(X(\R),\Q))&=\Hom_\Z(H_i(X(\R)),\Z)\otimes \Z/2\Z,\text{ and}\\
H^i(X(\R),\Z/2\Z)&=\Hom_\Z(H_i(X(\R)),\Z/2\Z)+\Ext_\Z(H_{i-1}(X(\R)),\Z/2\Z)\\
&=\Hom_\Z(H_i(X(\R)),\Z)\otimes \Z/2\Z
\;+\;
\Hom_\Z(H_i(X(\R))[2],\Z/2\Z)\\
&\qquad\qquad +\;
\Hom_\Z(H_{i-1}(X(\R))[2],\Z/2\Z).
\end{split}
\]
So applying $\tau$ to the left-hand side of \eqref{groth} gives
\be
\tau\bigg(\sum_i (-1)^i H^i(X(\R),\Q)\bigg)=\sum_i (-1)^i H^i(X(\R),\Z/2\Z).
\ee
On the right-hand side, the fact that $H_{2i+1}(X)$
vanishes and $H_{2i}(X)$ is torsion-free \cite[\S 5.2]{F} 
implies that $\tau(H^{2i}(X,\Q))=H^{2i}(X,\Z/2\Z)$.
The result now follows from Proposition \ref{comparison}.
\end{proof}
Examples abound to show that the restriction to odd-order elements is necessary
(consider $X=\bbP^1$, $\alpha:z\mapsto z^{-1}$).
Thus, for a group of even order acting on a toric variety, 
the equivariant Euler characteristic of the real locus cannot be simply
deduced from knowledge of the action on the cohomology groups of the complex variety. 

In the case where $X=\CT_W$ with $\alpha=w$ an element of $W$,
the left-hand side of Corollary \ref{oddorder} is computed by
Theorem \ref{main}, and a formula for the right-hand side may
be deduced from \cite[Theorem 1.1]{L}. The comparison gives nothing
new, and one would not expect it to, because the proofs of both theorems
use the same reduction to the case of a torus, and the analogue of
Corollary \ref{oddorder} for a torus $T$ is easy to prove. Namely, 
for any automorphism $\alpha$ of $T$ of finite odd order, one has
\be
\sum_i (-1)^i\, \tr(\alpha,H^i_c(T(\R),\Q))=
(-1)^{\dim T}\sum_i \tr(\alpha,H^{i}_c(T,\Q)),
\ee
because both sides equal $(-1)^{\dim T}\, 2^{\rank_\Z Y(T)^\alpha}$.

However, we wish to point out a curious complement to Corollary \ref{oddorder},
which appears to hold only in type A (except for odd-rank cases, where it
follows from Poincar\'e duality). For example, it fails for a Coxeter element in
type B$_2$, by \cite[Proposition 3.7(ii)]{L}.
\begin{proposition}
If $w\in\Sym_n$ has even order, then
\[ \sum_i (-1)^i\, \tr(w,H^{2i}(\CT_{\Sym_n},\Q))=0. \]
\end{proposition}
\begin{proof}
Recall the ring $\widehat{\bigoplus}_{n\geq 0}R(\Sym_n)$ from Section 3.
The following generating-function formula holds in 
$\Q[q]\otimes_\Q\widehat{\bigoplus}_{n\geq 0}R(\Sym_n)$ 
(see \cite[Theorem 6.2]{stembridge}):
\be \label{stembridge}
1+\sum_{n\geq 1}\sum_i q^i\,H^{2i}(\CT_{\Sym_n},\Q)
=\frac{\sum_{n\geq 0} 1_{\Sym_n}}{1-\sum_{n\geq 2}(q+q^2+\cdots+q^{n-1})1_{\Sym_n}}.
\ee
Equation \eqref{stembridge} can be deduced easily from \cite[Theorem 1.1]{L},
which implies that the left-hand side is the multiplicative inverse of 
$1-\sum_{n\geq 1}\gamma_{\Sym_n}$, where $\gamma_{\Sym_n}(w)=\det_{V}(q-w)$.
Setting $q=-1$ in \eqref{stembridge}, we deduce the following equation
in $\widehat{\bigoplus}_{n\geq 0}R(\Sym_n)$:
\be
1+\sum_{n\geq 1}\sum_i (-1)^i\,H^{2i}(\CT_{\Sym_n},\Q)
=\bigg(\sum_{n\geq 0} 1_{\Sym_n}\bigg).
\bigg(1+\sum_{\substack{n\geq 2\\n\text{ even}}}1_{\Sym_n}\bigg)^{-1}.
\ee
Extracting the degree-$n$ terms, we have an equality in $R(\Sym_n)$:
\be \label{othertypea}
\begin{split}
\sum_i (-1)^i\,&H^{2i}(\CT_{\Sym_n},\Q)\\
&=\sum_{m\geq 0}(-1)^m
\negthickspace\negthickspace
\sum_{\substack{n_1,n_2,\cdots,n_m\geq 2\\
n_1,n_2,\cdots,n_m\textup{ even}\\n_1+n_2+\cdots+n_m\leq n}}
\negthickspace\negthickspace
\Ind_{\Sym_{n-n_1-\cdots-n_m}\times \Sym_{n_1}\times
\cdots \times \Sym_{n_m}}^{\Sym_n}(1).
\end{split}
\ee 
Note that the right-hand side
of \eqref{othertypea} coincides with the right-hand side of Theorem
\ref{typea} when evaluated at any $w$ satisfying $\varepsilon(w)=1$
(in particular, since this holds for all elements of odd order,
Corollary \ref{oddorder} is visibly true.) 

To conclude the proof,
we must show that the right-hand side of \eqref{othertypea} takes the
value $0$ on any $w$ which has even order (i.e.\ contains an even cycle).
But the value in question is $\sum_{m\geq 0}(-1)^m f_m(w)$,
where $f_m(w)$ is the number of $m$-tuples
$(A_1,A_2,\cdots,A_m)$ of disjoint $w$-stable nonempty subsets of $\{1,\cdots,n\}$
such that each $|A_i|$ is even. Let $C\subseteq\{1,2,\cdots,n\}$ be a cycle of
$w$ such that $|C|$ is even, and let $y$ be the induced permutation of
$\{1,2,\cdots,n\}\setminus C$. Since any such $(A_1,A_2,\cdots,A_m)$ must have
either $C=A_i$ for some $i$, $C\subset A_i$ for some $i$, or $C\cap A_i=\emptyset$
for all $i$, we have
\be
f_m(w)=mf_{m-1}(y)+(m+1)f_m(y),
\ee
where $f_{-1}(y)=0$. It follows immediately that
$\sum_{m\geq 0}(-1)^m f_m(w)=0$.
\end{proof}

\end{document}